\documentclass[10pt]{amsart}
\usepackage{amsmath,amssymb,latexsym,cite}
\usepackage[small]{caption}
\usepackage{graphicx,color,mathrsfs,tikz}
\usepackage{subfigure,color}
\usepackage{cite}
\usepackage[colorlinks=true,urlcolor=blue,
citecolor=red,linkcolor=blue,linktocpage,pdfpagelabels,
bookmarksnumbered,bookmarksopen]{hyperref}
\usepackage[italian,english]{babel}
\usepackage{units}
\usepackage{enumitem}
\usepackage[left=2.6cm,right=2.6cm,top=2.75cm,bottom=2.75cm]{geometry}
\usepackage[hyperpageref]{backref}

\usepackage[colorinlistoftodos,prependcaption]{todonotes}

\numberwithin{equation}{section}
\newtheorem{theorem}{Theorem}[section]
\newtheorem{proposition}[theorem]{Proposition}
\newtheorem{lemma}[theorem]{Lemma}
\newtheorem{remark}[theorem]{Remark}

\newtheorem{definition}[theorem]{Definition}
\theoremstyle{definition}
\newtheorem*{theoBM}{Theorem}

\renewcommand{\epsilon}{\eps}
\renewcommand{\i}{{\rm i}}

\newcommand{\N}{{\mathbb N}}
\newcommand{\R}{{\mathbb R}}

\newcommand{\dvg}{{\rm div}}

\newcommand{\eps}{\varepsilon}

\newcommand{\iu}{{\rm i}}

\newcommand{\pnorm}[2][]{\if #1'' \left|#2\right|_p \else \left|#2\right|_{#1} \fi}

\newcommand{\scal}[2]{{#1} \cdot {#2}\,}

\newcommand{\loc}{{\rm loc}}

\newcommand{\C}{\mathbb{C}}

\renewcommand{\theta}{\vartheta}

\title[Maz'ya-Shaposhnikova limit in the magnetic setting]{The Maz'ya-Shaposhnikova limit in the magnetic setting}

\author[A.\ Pinamonti]{Andrea Pinamonti}
\author[M.\ Squassina]{Marco Squassina}
\author[E.\ Vecchi]{Eugenio Vecchi}

\address[A.\ Pinamonti]{Dipartimento di Matematica \newline\indent
	Universit\`a degli Studi di Padova,
	Via Trieste 63, 35121 Padova, Italy}
\email{Pinamont@math.unipd.it}

\address[M.\ Squassina]{Dipartimento di Matematica e Fisica \newline\indent
	Universit\`a Cattolica del Sacro Cuore,
	Via dei Musei 41, I-25121 Brescia, Italy}
\email{marco.squassina@dmf.unicatt.it}

\address[E.\ Vecchi]{Dipartimento di Matematica \newline\indent
	Universit\`a di Bologna,
	Piazza di Porta S. Donato 5, 40126, Bologna, Italy}
\email{eugenio.vecchi2@unibo.it}

\thanks{The authors are members of {\em Gruppo Nazionale per l'Analisi Ma\-te\-ma\-ti\-ca, la Probabilit\`a e le loro Applicazioni} (GNAMPA) of the {\em Istituto Nazionale di Alta Matematica} (INdAM). E.V. receives funding from the People Programme (Marie Curie Actions) of the European Union's Seventh
Framework Programme FP7/2007-2013/ under REA grant agreement No.\ 607643 (ERC Grant MaNET `Metric Analysis for Emergent Technologies')}



\begin{document}
\hyphenation{Spia-na-to}
\begin{abstract}
We prove a magnetic version of the Maz'ya-Shaposhnikova singular limit
of nonlocal norms with vanishing fractional parameter. This complements a general convergence 
result recently obtained by authors when the parameter approaches one.
\end{abstract}

\maketitle


\section{Introduction}

\noindent
About fifteen years ago,  V.\ Maz'ya and T.\ Shaposhnikova proved that for any $n\geq 1$ and $p\in [1,\infty)$, 
\[
\lim_{s\searrow 0} s\int_{\R^n}\int_{\R^n}\frac{|u(x)-u(y)|^p}{|x-y|^{n+ps}}dxdy= 
\frac{4\pi^{n/2}}{p\Gamma(n/2)} \|u\|^p_{L^p(\R^n)},
\]
whenever $u\in D^{s,p}_{0}(\R^n)$ for some $s\in (0,1)$. Here $\Gamma$ denotes the Gamma function and 
the space $D^{s,p}_{0}(\R^n)$ is the completion of $C^{\infty}_c(\R^n)$ with respect to the Gagliardo norm
$$
\int_{\R^n}\int_{\R^n}\frac{|u(x)-u(y)|^p}{|x-y|^{n+ps}}dxdy.
$$
Their motivation was basically that of complementing a previous
result by Bourgain-Brezis-Mironescu \cite{bourg,bourg2} providing new characterizations for functions in the Sobolev space $W^{1,p}(\Omega)$.
Precisely,  if $\Omega\subset\R^n$ is a smooth bounded domain, then for any $W^{1,p}(\Omega)$ there holds 
\begin{equation*}
\lim_{s\nearrow 1}(1-s)\int_{\Omega}\int_{\Omega}\frac{|u(x)-u(y)|^p}{|x-y|^{n+ps}}dxdy=
Q_{p,n}\int_{\Omega}|\nabla u|^pdx,
\end{equation*}
where $Q_{p,n}$ is defined by
\begin{equation}
\label{valoreK}
Q_{p,n}=\frac{1}{p}\int_{{\mathbb S}^{n-1}}|{\boldsymbol \omega}\cdot h|^{p}d\mathcal{H}^{n-1}(h),
\end{equation}
being ${\mathbb S}^{n-1}$ the unit sphere in $\R^n$
and ${\boldsymbol \omega}$ an arbitrary unit vector of $\R^n$. 
The above singular limits are natural and also admit a physical relevance
in the framework of the theory of Levy processes. Also, there is a
developed theory of fractional $s$-perimeters \cite{CRS}
and there have been several contributions concerning their
asymptotic analysis in the limits 
$s\nearrow 1$ and $s\searrow 0$ \cite{CV,AmbDepMart, Dip}.

One of the latest generalizations of this kind of convergence results appeared
recently in \cite{BM} in the context of {\em magnetic Sobolev spaces} $W_{A}^{1,2}(\Omega)$, see \cite{LL}.\ 
In fact, a relevant role in the study of particles which interact 
with a magnetic field $B=\nabla\times A$, $A:\R^n\to\R^n$,  is assumed by the {\em magnetic Laplacian} $(\nabla-\iu A)^2$ \cite{AHS,reed,LL}, yielding to  nonlinear Schr\"odinger equations of the type $- (\nabla-\iu A)^2 u + u = f(u),$
which have been extensively studied (see \cite{arioliSz} and the references therein). The operator is defined weakly as the differential of the energy
$$
W_{A}^{1,2}(\Omega)\ni u\mapsto \int_{\Omega}|\nabla u-\i A(x)u|^2dx.
$$
If $A:\R^n\to\R^n$ is a smooth field and $s \in (0,1)$,
a nonlocal magnetic counterpart of the magnetic laplacian,
\begin{equation*}
(-\Delta)^s_Au(x)=c(n,s) \lim_{\eps\searrow 0}\int_{B^c_\eps(x)}\frac{u(x)-e^{\i (x-y)\cdot A\left(\frac{x+y}{2}\right)}u(y)}{|x-y|^{n+2s}}dy,
\end{equation*}
where $c(n,s)$ is a normalization constant which behaves as follows
\begin{equation}
\label{asympt}
\lim_{s\searrow 0}\frac{c(n,s)}{s}=\frac{\Gamma(n/2)}{\pi^{n/2}},\qquad\,\, \lim_{s\nearrow 1}\frac{c(n,s)}{1-s}=\frac{2n\Gamma(n/2)}{\pi^{n/2}},
\end{equation}
was introduced  in \cite{piemar,I10} for complex-valued functions, with
motivations falling into
the framework of the general theory of L\'evy processes. Recently, the authors in \cite{nostro} (see \cite{BM} for $p=2$) proved that
if $A:\R^n\to \R^n$ is a $C^2$ vector field, then, for any $n\geq 1$, $p\in [1,\infty)$ and any 
	Lipschitz bounded domain $\Omega\subset\R^n$
	\begin{equation}
	\label{formula-s-1}
	\lim_{s\nearrow 1}(1-s)\int_{\Omega}\int_{\Omega}\frac{|u(x)-e^{\i (x-y)\cdot A\left(\frac{x+y}{2}\right)}u(y)|^p_p}{|x-y|^{n+ps}}dxdy=
	Q_{p,n}\int_{\Omega}|\nabla u-\i A(x)u|^p_p\, dx,
	\end{equation}
	for all $u\in W^{1,p}_A(\Omega)$, where $Q_{p,n}$ is as in \eqref{valoreK} and
	$|z|_p:=\left(|(\Re z_1,\ldots, \Re z_n)|^p+|(\Im z_1,\ldots, \Im z_n)|^p\right)^{1/p}.$ This has provided
	a new nonlocal characterization of the magnetic Sobolev spaces $W^{1,p}_A(\Omega).$
%
\vskip3pt
\noindent
The main goal of this paper is to complete the picture of \cite{nostro} by providing a {\em magnetic counterpart} of 
the convergence result by {\em Maz'ya-Shaposhnikova} for vanishing fractional orders $s$, namely for $s\searrow 0$.
\vskip3pt
\noindent
We consider a locally bounded vector potential field $A:\R^n\to\R^n$ and the space of complex valued functions $D^{s,p}_{A,0}(\R^n,\C)$ defined as the completion of $C^{\infty}_c(\R^n,\C)$ with respect to the norm
$$
\|u\|_{D^{s,p}_{A,0}}=\left(\int_{\R^n}\int_{\R^n}\frac{|u(x)-e^{\i (x-y)\cdot A\left(\frac{x+y}{2}\right)}u(y)|^p}{|x-y|^{n+ps}}dxdy\right)^{1/p}.
$$

\noindent
By combining Lemma~\ref{liminf} and Lemma~\ref{limsup}, we shall prove the following result. 

\begin{theorem}[Magnetic Maz'ya-Shaposhnikova]
	\label{main}
	Let $n\geq 1$ and $p\in [1,\infty)$. Then for every 
	$$
	u\in \bigcup_{0<s<1}D^{s,p}_{A,0}(\R^n,\C),
	$$
	there holds
	\[
	\lim_{s\searrow 0} s\int_{\R^n}\int_{\R^n}\frac{|u(x)-e^{\i (x-y)\cdot A\left(\frac{x+y}{2}\right)}u(y)|^p}{|x-y|^{n+ps}}dxdy= 
	\frac{4\pi^{n/2}}{p\Gamma(n/2)} \|u\|^p_{L^p(\R^n)}.
	\]
\end{theorem}

\vskip2pt
\noindent
In particular, while the singular limit as $s\nearrow 1$ generates the magnetic gradient $\nabla -\i A$, the limit for vanishing 
$s$ tends to destroy the magnetic effects yielding the $L^p(\R^n)$-norm of the function $u$.
We point out that, while in \eqref{formula-s-1} the norm of complex numbers is $|\cdot|_p$, in Theorem~\ref{main} we use the usual
norm $|\cdot|=|\cdot|_2$. In any case when $A=0$ and $u$ is real-valued the formulas are all consistent with the classical statements.
In the case $p=2$, combining the asymptotic formulas in \eqref{asympt} with Theorem~\ref{main} implies that
$$
\frac{c(n,s)}{2}\int_{\R^n}\int_{\R^n}\frac{|u(x)-e^{\i (x-y)\cdot A\left(\frac{x+y}{2}\right)}u(y)|^2}{|x-y|^{n+2s}}dxdy\thickapprox \|u\|^2_{L^2(\R^n)},
\qquad\text{as $s\searrow 0$,}
$$
for any $u\in D^{s,2}_0(\R^n)$ for some $s\in (0,1)$. Although the magnetic setting is mainly meaningful
in the framework of nonlocal Schr\"odinger equations, we remark that
for $E\subset\R^n$, if $E^c:=\R^n\setminus E$, the quantity
\begin{equation*}
	P_s(E;A):=
	\frac{1}{2}\int_{E}\int_{E}\frac{|1-e^{\i (x-y)\cdot A\left(\frac{x+y}{2}\right)}|}{|x-y|^{n+s}}dxdy+
	\int_{E}\int_{E^c}\frac{1}{|x-y|^{n+s}}dxdy 
\end{equation*}
plays the role of a nonlocal $s$-perimeter of $E$ depending on $A$, which
reduces for $A=0$ to the usual notion of fractional $s$-perimeter of $E\subset \R^n$.
Then, if ${\mathscr L}^n(E)$ denotes the $n$-dimensional Lebesgue measure of $E\subset\R^n$, Theorem~\ref{main}, 
applied with $p=1$ and $u(x)={\bf 1}_E(x)$, reads as 
$$
\lim_{s\searrow 0} sP_s(E,A)=\frac{4\pi^{n/2}}{\Gamma(n/2)} {\mathscr L}^n(E),
$$
provided that $P_{s_0}(E,A)<+\infty,$ for some $s_0\in (0,1)$.

\section{Proof of the main result}
\noindent
The proof of Theorem~\ref{main} follows by combining Lemma~\ref{liminf} and Lemma~\ref{limsup} below.

\begin{lemma}[Liminf inequality]
	\label{liminf}
	Let $n\geq 1$, $p\in [1,\infty)$ and let
	$$
	u\in \bigcup_{0<s<1}D^{s,p}_{A,0}(\R^n,\C). 
	$$
	Then 
	\[
	\liminf_{s\searrow 0} s\int_{\R^n}\int_{\R^n}\frac{|u(x)-e^{\i (x-y)\cdot A\left(\frac{x+y}{2}\right)}u(y)|^p}{|x-y|^{n+ps}}dxdy\geq 
		\frac{4\pi^{n/2}}{p\Gamma(n/2)} \|u\|^p_{L^p(\R^n)}.
	\]
\end{lemma}
\begin{proof}
	If 
	$$
		\liminf_{s\searrow 0} s\int_{\R^n}\int_{\R^n}\frac{|u(x)-e^{\i (x-y)\cdot A\left(\frac{x+y}{2}\right)}u(y)|^p}{|x-y|^{n+ps}}dxdy=\infty,
	$$
	the assertion follows. Otherwise, there exists a sequence $\{s_k\}_{k\in{\mathbb N}}\subset (0,1)$ with $s_k\searrow 0$ and 
	$$
	\liminf_{s\searrow 0} s\int_{\R^n}\int_{\R^n}\frac{|u(x)-e^{\i (x-y)\cdot A\left(\frac{x+y}{2}\right)}u(y)|^p}{|x-y|^{n+ps}}dxdy=
	\lim_{k\to\infty} s_k\int_{\R^n}\int_{\R^n}\frac{|u(x)-e^{\i (x-y)\cdot A\left(\frac{x+y}{2}\right)}u(y)|^p}{|x-y|^{n+ps_k}}dxdy,
	$$
	the limit being finite. For a.e.\ $x,y\in \R^n$ we have the Diamagnetic inequality (cf.\ \cite[Remark 3.2]{piemar})
	\begin{equation}
	\label{diam}
	||u(x)|-|u(y)||\leq
	|u(x)-e^{\i (x-y)\cdot A\left(\frac{x+y}{2}\right)}u(y)|.
	\end{equation}
In particular, since $u\in D^{s_k,p}_{A,0}(\R^n,\C)$, 
we have $|u|\in D^{s_k,p}_0(\R^n)$ and, for any $k\geq 1$,
$$
s_k\int_{\R^n}\int_{\R^n}\frac{||u(x)|-|u(y)||^p}{|x-y|^{n+ps_k}}dxdy
\leq
s_k\int_{\R^n}\int_{\R^n}\frac{|u(x)-e^{\i (x-y)\cdot A\left(\frac{x+y}{2}\right)}u(y)|^p}{|x-y|^{n+ps_k}}dxdy.
$$
Taking the limit as $k\to\infty$ on both sides and invoking 
\cite[Theorem 3]{mazia} applied to $|u|,$ yields
$$
\frac{4\pi^{n/2}}{p\Gamma(\frac{n}{2})} \||u|\|^p_{L^p(\R^n)}\leq
	\lim_{k\to\infty}  s_k\int_{\R^n}\int_{\R^n}\frac{|u(x)-e^{\i (x-y)\cdot A\left(\frac{x+y}{2}\right)}u(y)|^p}{|x-y|^{n+ps_k}}dxdy,
$$
which concludes the proof.
\end{proof}

\begin{remark}[Magnetic Hardy inequality] \rm 
By combining the pointwise Diamagnetic inequality \eqref{diam} 
with the fractional Hardy inequality \cite{frank}, for $n>ps$ the following {\em magnetic Hardy 
inequality} holds: there exists a positive constant ${\mathcal H}_{n,s,p}$ such that
\begin{equation}
\label{hardym}
\int_{\R^n} \frac{|u(x)|^p}{|x|^{sp}}dx\leq {\mathcal H}_{n,s,p}
\int_{\R^n}\int_{\R^n}\frac{|u(x)-e^{\i (x-y)\cdot A\left(\frac{x+y}{2}\right)}u(y)|^p}{|x-y|^{n+ps}}dxdy,
\end{equation}
for every $u\in D^{s,p}_{A,0}(\R^n,\C)$. Similarly the following
{\em magnetic Sobolev inequality} holds: there exists a positive constant ${\mathcal S}_{n,s,p}$ such that
\begin{equation*}
	\Big(\int_{\R^n} |u(x)|^{\frac{np}{n-sp}}dx\Big)^{\frac{n-sp}{n}}\!\!\!\!\!\leq {\mathcal S}_{n,s,p}
	\int_{\R^n}\int_{\R^n}\frac{|u(x)-e^{\i (x-y)\cdot A\left(\frac{x+y}{2}\right)}u(y)|^p}{|x-y|^{n+ps}}dxdy,
\end{equation*}
for every $u\in D^{s,p}_{A,0}(\R^n,\C)$.
\end{remark}

\noindent
Next we state a second lemma completing the proof of Theorem~\ref{main} when combined with Lemma~\ref{liminf}.

\begin{lemma}[Limsup inequality]
	\label{limsup}
		Let $n\geq 1$, $p\in [1,\infty)$ and let
		$$
		u\in \bigcup_{0<s<1}D^{s,p}_{A,0}(\R^n,\C). 
		$$
	Then
\[
\limsup_{s\searrow 0} 
s\int_{\R^n}\int_{\R^n}\frac{|u(x)-e^{\i (x-y)\cdot A\left(\frac{x+y}{2}\right)}u(y)|^p}{|x-y|^{n+ps}}dxdy
\leq 	\frac{4\pi^{n/2}}{p\Gamma(n/2)}  \|u\|^p_{L^p(\R^n)}.
\]
\end{lemma}
\begin{proof}
If $u\not\in L^p(\R^n),$ there is nothing to prove. Hence, we may assume that $u\in L^p(\R^n)$. 
We observe that 
\begin{align*}
& s\int_{\R^n}\int_{\R^n}\frac{|u(x)-e^{\i (x-y)\cdot A\left(\frac{x+y}{2}\right)}u(y)|^p}{|x-y|^{n+ps}}dxdy  \\ &=s\int_{\R^n}\int_{\{|x|\leq |y| \leq 2|x|\}}\frac{|u(x)-e^{\i (x-y)\cdot A\left(\frac{x+y}{2}\right)}u(y)|^p}{|x-y|^{n+sp}} dxdy  \notag\\
\nonumber
&+s\int_{\R^n}\int_{\{|y|\geq 2|x|\}}\frac{|u(x)-e^{\i (x-y)\cdot A\left(\frac{x+y}{2}\right)}u(y)|^p}{|x-y|^{n+sp}} dxdy \\
\nonumber
&+s\int_{\R^n}\int_{\{|x|\geq |y|\}}\frac{|u(x)-e^{\i (x-y)\cdot A\left(\frac{x+y}{2}\right)}u(y)|^p}{|x-y|^{n+sp}} dxdy\\
\nonumber
&= 2s\int_{\R^n}\int_{\{|x|\leq |y|\leq 2|x|\}}\frac{|u(x)-e^{\i (x-y)\cdot A\left(\frac{x+y}{2}\right)}u(y)|^p}{|x-y|^{n+sp}} dxdy \\
\nonumber
&+2s\int_{\R^n}\int_{\{|y|\geq 2|x|\}}\frac{|u(x)-e^{\i (x-y)\cdot A\left(\frac{x+y}{2}\right)}u(y)|^p}{|x-y|^{n+sp}} dxdy,
\end{align*}
where the last equality follows noticing that since $|e^{\i (x-y)\cdot A\left(\frac{x+y}{2}\right)}|=1$ then
\begin{align*}
\int_{\R^n}\int_{\{|x|\geq |y|\}}\frac{|u(x)-e^{\i (x-y)\cdot A\left(\frac{x+y}{2}\right)}u(y)|^p}{|x-y|^{n+sp}} dxdy&=\int_{\R^n}\int_{\{|y|\geq |x|\}}\frac{|u(x)-e^{\i (x-y)\cdot A\left(\frac{x+y}{2}\right)}u(y)|^p}{|x-y|^{n+sp}} dxdy\\
&=\int_{\R^n}\int_{\{|y|\geq 2|x|\}}\frac{|u(x)-e^{\i (x-y)\cdot A\left(\frac{x+y}{2}\right)}u(y)|^p}{|x-y|^{n+sp}} dxdy \\
&+\int_{\R^n}\int_{\{|x|\leq |y|\leq 2|x|\}}\frac{|u(x)-e^{\i (x-y)\cdot A\left(\frac{x+y}{2}\right)}u(y)|^p}{|x-y|^{n+sp}} dxdy.
\end{align*}
Using the triangle inequality for the $L^p$-norm on $\R^{2n}$ and
recalling that $|e^{\i (x-y)\cdot A\left(\frac{x+y}{2}\right)}|=1,$ yields
\begin{align*}
&s\int_{\R^n}\int_{\{|y|\geq 2|x|\}}\frac{|u(x)-e^{\i (x-y)\cdot A\left(\frac{x+y}{2}\right)}u(y)|^p}{|x-y|^{n+sp}} dxdy & \\
\nonumber
&\leq \left\{\left(s\int_{\R^n}\int_{\{|y|\geq 2|x|\}}\frac{|u(x)|^p}{|x-y|^{n+sp}} dxdy\right)^{1/p}+ \left(s\int_{\R^n}\int_{\{|y|\geq 2|x|\}}\frac{|u(y)|^p}{|x-y|^{n+sp}}dxdy\right)^{1/p}\right\}^p.
\end{align*}
We claim that 
\begin{align*}
\lim_{s\searrow 0} s\int_{\R^n}\int_{\{|y|\geq 2|x|\}}\frac{|u(y)|^p}{|x-y|^{n+sp}}dxdy=0.
\end{align*}
Observe that $2|x-y|\geq |y|+ (|y|-2|x|)$.
Then, if $|y|\geq 2|x|$ we get $2|x-y|\geq |y|$. Now, if ${\mathscr H}^{n-1}$ denotes the $(n-1)$-dimensional
Haudorff measure, it follows that 
\begin{align*}
 s^{1/p}\left(\int_{\R^n}\int_{\{|y|\geq 2|x|\}}\frac{|u(y)|^p}{|x-y|^{n+sp}}dx dy\right)^{1/p} &\leq s^{1/p}\left(2^{n+sp}\int_{\R^n}\frac{|u(y)|^p}{|y|^{n+sp}}\Big(\int_{\{|x|\leq |y|/2\}} dx\Big)dy\right)^{1/p}\\
&=2^s\left(\frac{s}{n} {\mathscr H}^{n-1}({\mathbb S}^{n-1})\right)^{1/p} \left(\int_{\R^n} \frac{|u(y)|^p}{|y|^{sp}}dy\right)^{1/p},
\end{align*}
and the last term goes to zero as $s\searrow 0$. Notice that $y\mapsto |y|^{-s} u(y)$ remains bounded 
in $L^p(\R^n)$ as $s\searrow 0$ by the argument indicated here below. Observe now that, 
if $|y|\geq 2|x|$ we then get $|x-y|\geq |x|$ yielding
\begin{align*}
&\left(s\int_{\R^n}\int_{\{|y|\geq 2|x|\}}\frac{|u(x)|^p}{|x-y|^{n+sp}} dxdy\right)^{1/p}\leq 
\left(s\int_{\R^n}\int_{\{|x-y|\geq |x|\}}\frac{|u(x)|^p}{|x-y|^{n+sp}} dxdy\right)^{1/p}\\
&=\left(s\int_{\R^n}|u(x)|^p\int_{B(0,|x|)^c}\frac{dz}{|z|^{n+sp}} dx\right)^{1/p}=
\frac{{\mathscr H}^{n-1}({\mathbb S}^{n-1})^{1/p}}{p^{1/p}} \left(\int_{\R^n} \frac{|u(x)|^p}{|x|^{sp}}dx\right)^{1/p}.
\end{align*}
Moreover $|x|^{-sp}|u(x)|^p=f_s(x)+g_s(x)$, where 
$$
f_s(x):=\frac{|u(x)|^p}{|x|^{sp}}{\bf 1}_{B(0,1)}(x),\qquad
g_s(x):=\frac{|u(x)|^p}{|x|^{sp}}{\bf 1}_{B(0,1)^c}(x)\leq|u(x)|^p{\bf 1}_{B(0,1)^c}(x)\in L^1(\R^n),
$$
and $s\mapsto f_s$ is decreasing and, moreover, by the Hardy inequality \eqref{hardym}
and the assumption on $u$, it follows that $f_{\tilde s}\in L^1(\R^n)$ for some $\tilde s\in (0,1)$. Hence, by monotone 
and dominated convergence, we conclude that
\begin{equation*}
\limsup_{s\searrow 0} s\int_{\R^n}\int_{\{|y|\geq 2|x|\}}\frac{|u(x)|^p}{|x-y|^{n+sp}} dxdy\leq \frac{{\mathscr H}^{n-1}({\mathbb S}^{n-1})}{p}\|u\|_{L^p(\R^n)}^p=
\frac{2\pi^{n/2}}{p\Gamma(\frac{n}{2})}\|u\|_{L^p(\R^n)}^p.
\end{equation*}
Then, we conclude from the above inequalities that
\begin{equation}
\label{seconda}
\limsup_{s\searrow 0} 2s\int_{\R^n}\int_{\{|y|\geq 2|x|\}}\frac{|u(x)-e^{\i (x-y)\cdot A\left(\frac{x+y}{2}\right)}u(y)|^p}{|x-y|^{n+sp}} dxdy 
\leq \frac{4\pi^{n/2}}{p\Gamma(\frac{n}{2})}\|u\|_{L^p(\R^n)}^p.
\end{equation}
We claim that 
\begin{equation}
\label{claim}
\limsup_{s \searrow 0} 2s  \int_{\R^n}\int_{\{|x|\leq |y|\leq 2|x|\}} \dfrac{|u(x)-e^{\i (x-y)\cdot A\left(\frac{x+y}{2}\right)}u(y)|^p}{|x-y|^{n+sp}}\, dxdy = 0.
\end{equation}
By assumption let $\tau \in (0,1)$ such that $u \in D^{\tau,p}_{A,0}(\R^n)$. 
Now let $N \geq 1$ and $s<\tau$. Then
\begin{equation*}
\begin{aligned}
2s \, &\int_{\R^n} \int_{\{|x|\leq |y|\leq 2|x|\}}\dfrac{|u(x)-e^{\i (x-y)\cdot A\left(\frac{x+y}{2}\right)}u(y)|^p}{|x-y|^{n+sp}}\, dxdy \\
&=2s \, \int_{\R^n} \int_{\underset{\{|x|\leq |y|\leq2|x|\}}{\{|x-y|\leq N\}}}\dfrac{|u(x)-e^{\i (x-y)\cdot A\left(\frac{x+y}{2}\right)} u(y)|^p}{|x-y|^{n+sp}}\, dxdy  \\
&+ 2s \, \int_{\R^n} \int_{\underset{\{|x|\leq|y|\leq 2|x|\}}{\{|x-y|> N\}}}\dfrac{|u(x)-e^{\i (x-y)\cdot A\left(\frac{x+y}{2}\right)} u(y)|^p}{|x-y|^{n+sp}}\, dxdy
=: \mathcal{I} + \mathcal{II}.
\end{aligned}
\end{equation*}
Let us consider $\mathcal{I}$ first. Since $|x-y| \leq N$, it holds that
$$
\dfrac{1}{|x-y|^{n+sp}} = \dfrac{|x-y|^{p(\tau -s)}}{|x-y|^{n+\tau p}} \leq \dfrac{N^{p(\tau -s)}}{|x-y|^{n+\tau p}}.
$$
Therefore $\mathcal{I}$ goes to zero as $s\searrow$, since
\begin{equation*}
\mathcal{I} \leq 2s N^{p(\tau -s)}\, \int_{\R^n} \int_{\underset{\{|x|\leq |y|\leq 2|x|\}}{\{|x-y|\leq N\}}}\dfrac{|u(x)-e^{\i (x-y)\cdot A\left(\frac{x+y}{2}\right)}u(y)|^p}{|x-y|^{n+\tau p}}\,dx dy.
\end{equation*}
Let us now move to $\mathcal{II}$. Since $|u(x)-e^{\i (x-y)\cdot A\left(\frac{x+y}{2}\right)}u(y)|^{p} \leq 2^{p-1} \left(|u(x)|^{p} + |u(y)|^{p}\right)$, we get
\begin{equation*}
\mathcal{II} \leq 2^p s  \int_{\R^n}\!\int_{\underset{\{ |x|\leq |y|\leq 2|x|\}}{\{|x-y|\geq N\}}} \!\dfrac{|u(x)|^{p}}{|x-y|^{n+sp}} dxdy +
 2^p s \int_{\R^n}\!\int_{\underset{\{ |x|\leq |y|\leq 2|x|\}}{\{|x-y|\geq N\}}} \!\dfrac{|u(y)|^{p}}{|x-y|^{n+sp}} dxdy
=: \mathcal{II'} + \mathcal{II''}.
\end{equation*}
Regarding $\mathcal{II'}$, since $|x-y|\geq N$ and $|y|\leq 2|x|$, it holds 
$$
N \leq |x-y| \leq |x| + |y| \leq 3|x|,
$$
which implies that $|x| \geq \tfrac{N}{3}$. In particular, this also implies that
$$
\mathcal{II'} \leq 2^p s \int_{\{|x| \geq N/3\}}\left(\int_{\{|x-y|\geq N\}} \dfrac{|u(x)|^{p}}{|x-y|^{n+sp}}dy\right) dx\leq C(n,p) 
\int_{\{|x| \geq N/3\}} |u(x)|^{p} \, dx.
$$
For $\mathcal{II''}$, since as before $|x-y|\geq N$ and $|x|\leq |y|$, we have 
$$
N \leq |x-y| \leq |x| + |y| \leq 2|y|,
$$
which implies $|y| \geq \dfrac{N}{2} \geq \dfrac{N}{3}.$ Therefore, we get
\begin{equation*}
\mathcal{II''}  \leq 2^ps \int_{\{|y|\geq N/3\}}|u(y)|^p\left(\int_{\{|z|\geq N\}} \dfrac{1}{|z|^{n+sp}}dz\right) dy\leq
C(n,p) \, \int_{\left\{|y| \geq N/3 \right\}} |u(y)|^{p} \, dy.
\end{equation*}
Combining the estimates for $\mathcal{II'}$ and $\mathcal{II''}$, we get 
$$
\mathcal{II} \leq C(n,p) \, \int_{\left\{|x| \geq N/3 \right\}} |u(x)|^{p}  dx,
$$
which is a bound independent of $s$. Now, going back to 
$$
\limsup_{s \searrow 0} 2s \int_{\R^n} \int_{\{|x|<|y|<2|x|\}}\dfrac{|u(x)-e^{\i (x-y)\cdot A\left(\frac{x+y}{2}\right)}u(y)|^p}{|x-y|^{n+sp}}\,dxdy
\leq 2C(n,p)\|u\|^p_{L^p(B(0,N/3)^c)},
$$
and \eqref{claim} follows letting $N\to\infty$, since $u\in L^p({\mathbb R}^n)$. 
Collecting \eqref{seconda} and \eqref{claim}, the assertion follows. 
\end{proof}



\bigskip

\begin{thebibliography}{99}
	
%
	\bibitem{AmbDepMart}
	L.\ Ambrosio, G.\ De Philippis, L.\ Martinazzi,
	{\it $\Gamma-$convergence of nonlocal perimeter functionals,}
	Manuscripta Math. {\bf 134} (2011), 377--403.
%
%
	
	\bibitem{arioliSz}
	G.\ Arioli, A.\ Szulkin,
	{\it A semilinear Schr\"odinger equation in the presence of a magnetic field}, Arch. Ration. Mech. Anal. {\bf 170} (2003), 277--295.
	
	\bibitem{AHS}
	J. Avron, I. Herbst, B. Simon, {\it Schr\"odinger operators with magnetic fields. I. General interactions}, Duke Math. J. {\bf 45} (1978), 847--883.
	
	\bibitem{bourg}
	J. Bourgain, H. Brezis, P. Mironescu, 
	{\it Another look at Sobolev spaces},
	in \emph{Optimal Control and Partial Differential Equations. A Volume in Honor of Professor Alain Bensoussan's 60th Birthday}
	(eds. J. L. Menaldi, E. Rofman and A. Sulem), IOS Press, Amsterdam, 2001, 439--455.
	
	\bibitem{bourg2}
	J. Bourgain, H. Brezis, P. Mironescu,
	{\it Limiting embedding theorems for $W^{s,p}$ when $s \uparrow 1$ and applications},
	{J. Anal. Math.} \textbf{87} (2002), 77--101.
	
	
%
%
%
	
	\bibitem{CRS}
	L.\ Caffarelli, J.\ -M.\ Roquejoffre, O.\ Savin,
	{\it Nonlocal minimal surfaces},
	Comm.\ Pure Appl. Math. {\bf 63} (2010), 1111--1144.
	
	\bibitem{CV}
	L.\ Caffarelli, E.\ Valdinoci, 
	{\em Regularity properties of nonlocal 
	minimal surfaces via limiting arguments}, Adv. Math. {\bf 248} (2013), 843--871. 
	
	\bibitem{piemar}
	P.\ d'Avenia, M.\ Squassina,  
	{\it Ground states for fractional magnetic operators},  preprint,
	\url{http://arxiv.org/abs/1601.04230} 
	
	
	\bibitem{Dip}
	S.\ Dipierro, A.\ Figalli, G.\ Palatucci, E.\ Valdinoci, 
	{\it Asymptotics of the s-perimeter as $s \to 0$,} 
	Discrete Contin. Dyn. Syst. {\bf 33} (2013), 2777--2790.

\bibitem{frank}
R.L.\ Frank, R.\ Seiringer, {\em Non-linear ground state representation 
and sharp Hardy inequalities},  J. Funct.\ Anal. {\bf 255} (2008), 3407--3430.
	
	\bibitem{I10}
	T.\ Ichinose, {\it Magnetic relativistic Schr\"odinger operators and 
		imaginary-time path integrals}, Mathematical physics, spectral theory 
	and stochastic analysis, 247--297, Oper.\ Theory Adv.\ Appl. {\bf 232}, Birkh\"auser/Springer, Basel, 2013.
	
	\bibitem{LL}
	E.\ Lieb and M.\ Loss, Analysis, Graduate Studies in Mathematics {\bf 14}, 2001.
	
	\bibitem{mazia}
	V.\ Maz'ya and T.\ Shaposhnikova,
	{\it On the Bourgain, Brezis, and Mironescu theorem concerning limiting embeddings of fractional Sobolev spaces},
	J. Funct. Anal. \textbf{195} (2002), 230--238.
	
	
	
	
	
	\bibitem{nostro}
	A.\ Pinamonti, M.\ Squassina, E.\ Vecchi, 
	{\em Magnetic BV functions and the Bourgain-Brezis-Mironescu formula}, preprint,
	\url{https://arxiv.org/abs/1609.09714} 
	
	\bibitem{reed}
	M.\ Reed, B.\ Simon, Methods of modern mathematical 
	physics, I, Functional analysis, Academic Press, Inc.,
	New York, 1980
	
	\bibitem{BM}
	M.\ Squassina, B.\ Volzone, 
	{\it Bourgain-Brezis-Mironescu formula for magnetic operators,}
	C. R. Math. Acad. Sci. Paris {\bf 354} (2016), 825--831.
	
	
\end{thebibliography}
\end{document}